\documentclass[11pt]{amsart}
\usepackage{FrankTall}
\usepackage[usenames,dvipsnames]{xcolor}
\usepackage{xspace}
\newcommand{\arhan}{Arhan\-gel'ski\u{\i}\xspace}
\newcommand{\groth}{Gro\-then\-dieck\xspace}
\newcommand{\lind}{Lindel\"of\xspace}
\newcommand{\slind}{surlindel\"of\xspace}

\newcommand{\frech}{Fr\'e\-chet-Ury\-sohn\xspace}
\newcommand{\MA}{\ensuremath{\mathrm{MA}}}
\newcommand{\PFA}{\ensuremath{\mathrm{PFA}}}
\newcommand{\ZFC}{\ensuremath{\mathrm{ZFC}}}
\newcommand{\CH}{\ensuremath{\mathrm{CH}}}
\renewcommand{\diamond}{\diamondsuit}
\newcommand{\refmap}[2]{\ensuremath{{#1}_{{}_{#2}}}} %reflection map

\newcommand\blfootnote[1]{%
  \begingroup
  \renewcommand\thefootnote{}\footnote{#1}%
  \addtocounter{footnote}{-1}%
  \endgroup
}

\makeatletter
\def\@setthanks{\vspace{-\baselineskip}\def\thanks##1{\@par##1\@addpunct.}\thankses}
\makeatother

\begin{document}
\title[Countable Tightness \& Grothendieck]{Countable Tightness and the Grothendieck Property in $\mathbf{C_p}$-theory}
\author[Franklin D.~Tall]{Franklin D.~Tall$^{1}$}
\thanks{$^{1}$Research supported by NSERC Grant A-7354.}
\date{\today}
\maketitle


\begin{abstract}
The \groth property has become important in research on the definability of pathological Banach spaces \cite{Casazza}, \cite{Hamelb}, and especially \cite{Hamel}. We here answer a question of \arhan by proving it undecidable whether countably tight spaces with \lind finite powers are \groth. We answer another of his questions by proving that PFA implies \lind countably tight spaces are \groth. We also prove that various other consequences of $\MA_{\omega_1}$ and $\PFA$ considered by \arhan, Okunev, and Reznichenko are not theorems of $\ZFC$.
\end{abstract}

%%%%%%%%%%%%%%%%%%%%%%%%%%%%%%%%%%%%%%%%
%%%%%%%%%%%%%% INTRODUCTION %%%%%%%%%%%%
%%%%%%%%%%%%%%%%%%%%%%%%%%%%%%%%%%%%%%%%
\section{Introduction}
For a topological space $X$, $C_p(X)$ is the set of continuous real-valued functions on $X$, given the pointwise topology inherited from $\reals^X$. The classic theorem of \groth \cite{Grothendieck1952} states:

\begin{prop}
\label{prop:groth}
	Let $X$ be countably compact and let $A \subseteq C_p(X)$ be such that every infinite subset of $A$ has a limit point in $C_p(X)$. Then the closure of $A$ in $C_p(X)$ is compact.
\end{prop}

This theorem has many applications in Analysis. We became interested in it due to its applications in Model Theory (see \cite{Casazza}, \cite{Hamelb}, and \cite{Hamel}). These involve questions of definability, especially of pathological Banach spaces. The upshot is that if certain topological spaces (\emph{type spaces}) associated with a \emph{logic} have the closure property $X$ has in the above theorem, then these Banach spaces are not definable in that logic. We are therefore interested in what classes of topological spaces other than the countably compact ones satisfy the conclusion of Proposition \ref{prop:groth}. We restrict ourselves to only consider infinite completely regular spaces. 

\begin{defir}[\cite{Arhangelskii1998}]
	$A \subseteq X$ is \emph{countably compact in $X$} if every infinite subset of $A$ has a limit point in $X$.
	$X$ is a \emph{$g$-space} if each $A \subseteq X$ which is countably compact in $X$ has compact closure.
	$X$ is a \emph{\groth space} (resp.~\emph{weakly \groth space}) if $C_p(X)$ is a hereditary $g$-space (resp.~a $g$-space).
\end{defir}

\begin{defi}
	$X$ is \emph{countably tight} if whenever $A \subseteq X$ and $x \in \clsr{A}$, there is a countable $B \subseteq A$ such that $x \in \clsr{B}$.
	$X$ is \emph{realcompact} if $X$ can be embedded as a closed subspace of a product of copies of the real line.
\end{defi}

\begin{thmr}[\cite{Arhangelskii1998}]
\label{thm:clear}
	If $X$ is countably tight, then $X$ is weakly \groth.
\end{thmr}

This is stated as ``clear'' in \cite{Arhangelskii1998}. Here is a proof:

Clearly, 

\begin{lem}
	A closed subspace of a realcompact space is realcompact.
\end{lem}

\begin{lemr}[\cite{Engelking1989}]
	A completely regular space is compact if and only if it is realcompact and countably compact.
\end{lemr}

\begin{defi}
	A space is \emph{wD} if whenever $\seq{d}{n}{n<\omega}$ is a closed discrete subspace, there is an infinite $S \subseteq \omega$ and a discrete collection of open sets $\st{U_n}{n \in S}$ with $d_n \in U_n$ for all $n \in S$.
\end{defi}

\begin{lemr}[\cite{vanDouwen1984}, \cite{Vaughan1978}]
Every realcompact space is wD.
\end{lemr}

\begin{lem}[folklore]
	Let $X$ be wD. Let $Y$ be countably compact in $X$. Then $\clsr{Y}$ is countably compact.
\end{lem}

\begin{proof}
	Suppose not. Let $\seq{d}{n}{n<\omega}$ be a closed discrete subspace of $\clsr{Y}$. Let $\st{U_n}{n \in S}$ be a discrete collection of open subsets of $X$, with $d_n \in U_n$ for every $n \in S$, where $S \subseteq \omega$ is infinite. Pick $e_n \in U_n \insect Y$. Then $\st{e_n}{n \in S}$ is a closed discrete subspace of $Y$, contradiction.
\end{proof}

\begin{lemr}[\cite{ArhangelskiiFunction}]
	If $X$ is countably tight, then $C_p(X)$ is realcompact.
\end{lemr}

\begin{proof}[Proof of Theorem \ref{thm:clear}]
	Let $X$ be countably tight. Then $C_p(X)$ is realcompact and hence wD. Let $Y$ be countably compact in $C_p(X)$. Then $\clsr{Y}$ is countably compact. But $\clsr{Y}$ is realcompact, so $\clsr{Y}$ is compact.
\end{proof}

%%%%%%%%%%%%%%%%%%%%%%%%%%%%%%%%%%%%%%%%
%%%%%%%%%%%%%%%%%% PFA %%%%%%%%%%%%%%%%%
%%%%%%%%%%%%%%%%%%%%%%%%%%%%%%%%%%%%%%%%
\section{Applications of the Proper Forcing Axiom}
In \cite{Arhangelskii1998}, \arhan proved:

\begin{prop}
\label{thm:MAnoCH}
	$\MA$ + $\neg\CH$ implies that if $X$ is countably tight and $X^n$ is \lind for all $n < \omega$, then $X$ is \groth.
\end{prop}

In fact, $\MA_{\omega_1}$ suffices.

A dramatic strengthening of Proposition \ref{thm:MAnoCH} is

\begin{thm}
\label{thm:dramatic}
	$\PFA$ implies \lind countably tight spaces are \groth.
\end{thm}

\begin{proof}
This actually follows easily from known results. First, a definition:

\begin{defi}
	A space is \emph{\slind} if it is a subspace of $C_p(X)$ for some \lind $X$.
\end{defi}

\arhan \cite{ArhangelskiiFunction} proved: 

\begin{lem}\label{lem:PFAarhan}
	$\PFA$ implies that every \slind compact space is countably tight.
\end{lem}

Okunev and Reznichenko \cite{Okunev2007} proved:

\begin{lem}\label{lem:OR17}
	$\MA_{\omega_1}$ implies that every separable \slind compact countably tight space is metrizable.
\end{lem}

\begin{defi}
A space $X$ is \emph{Fr\'echet-Urysohn} if whenever $x$ is a limit point of $Z \subseteq X$, there is a sequence in $Z$ converging to $x$. 
\end{defi}

It follows quickly that:

\begin{thm}
	$\PFA$ implies that every \slind compact space is \frech.
\end{thm}

\begin{proof}
	Metrizable spaces are clearly \frech. By countable tightness, if $K$ is compact and $L \subseteq K$ and $p \in \clsr{L}$, then there is a countable $M \subseteq L$ such that $p \in \clsr{M}$. But $\clsr{M}$ is separable compact and so metrizable.
\end{proof}

\arhan proved:

\begin{lemr}[\cite{Arhangelskii1998}]
\label{lem:frech}
	$X$ is \groth if and only if it is weakly \groth and compact subspaces of $C_p(X)$ are Fr\'echet-Urysohn.
\end{lemr}

This proves Theorem \ref{thm:dramatic}.
\end{proof}

Okunev and Reznichenko \cite{Okunev2007} point out that the conclusions of Lemmas \ref{lem:PFAarhan} and \ref{lem:OR17} can be simultaneously consistently achieved without large cardinals, so we have:

\begin{thm}
	If $\ZFC$ is consistent, so is $\ZFC$ plus ``every \lind countably tight space is \groth''.
\end{thm}

Lemmas \ref{lem:PFAarhan} and \ref{lem:OR17} actually consistently solve several other problems of \arhan:

\begin{probr}[\cite{Arhangelskii1998}]
\label{arhanprob:sepcompact}
	If $X$ is separable and compact and $Y \subseteq C_p(X)$ is \lind, does $Y$ have a countable network?
\end{probr}

\begin{probr}[\cite{ArhangelskiiFunction}]
\label{arhanprob:sepcompactlind}
	If $X$ is separable and compact and $C_p(X)$ is \lind, must $X$ be hereditarily separable?
\end{probr}

Notice that a positive answer to the first of these yields a positive answer to the second, since a space with a countable network is clearly hereditarily separable.

\begin{lemr}[{\cite[I.1.3]{ArhangelskiiFunction}}]
	$X$ has a countable network if and only if $C_p(X)$ does.
\end{lemr}

Okunev \cite{Okunev1995} considers versions of Problem \ref{arhanprob:sepcompact} with the additional hypothesis that finite powers of $Y$ are \lind. He proves:

\begin{prop}
\label{prop:okunevproves}
	$\MA + \neg\CH$ implies that if $Y$ is a space with all finite powers \lind and $X$ is a separable compact subspace of $C_p(Y)$, then $X$ is metrizable.
\end{prop}

He states that this is a reformulation of

\begin{prop}
\label{prop:okunevasserts}
	$\MA + \neg\CH$ implies that if $X$ is a separable compact space and $Y \subseteq C_p(X)$ has all finite powers \lind, then $Y$ has a countable network.
\end{prop}

Okunev and Reznichenko note that actually $\MA_{\omega_1}$ suffices for these instead of $\MA + \neg\CH$. Okunev and Reznichenko also prove:

\begin{propr}[{\cite[1.8]{Okunev2007}}]\label{prop:slind.cpct.sep.metrizable}
	$\PFA$ implies that every \slind compact separable space is metrizable.
\end{propr}

\begin{propr}[{\cite[1.9]{Okunev2007}}]\label{prop:slind.cpct.monolithic}
	$\PFA$ implies every \slind compact space is $\aleph_0$-monolithic, where a space is \emph{$\aleph_0$-monolithic} if the closure of every countable set has countable network weight. 
\end{propr}

We can use Lemmas \ref{lem:PFAarhan} and \ref{lem:OR17} to prove:

\begin{thm}\label{thm:uselemmas}
	$\PFA$ implies that if $X$ is a separable compact space and $Y \subseteq C_p(X)$ is \lind, then $Y$ has a countable network.
\end{thm}

\begin{proof}
	We closely follow part of the argument in \cite{Okunev1995} for Proposition \ref{prop:okunevasserts}. He starts by recalling some material from \cite{ArhangelskiiFunction} (or see \cite{Tkachuk2011}). Given a continuous map $p: X \to Y$, the \emph{dual map} $p^*: C_p(Y) \to C_p(X)$ is defined by $p^*(f) = f\com p$, for all $f \in C_p(Y)$. The dual map is always continuous; it is an embedding if and only if $p$ is onto. If $Y \subseteq C_p(X)$, then the \emph{reflection map} $\refmap{\vphi}{XY}:X \to C_p(Y)$ is defined by $\refmap{\vphi}{XY}(x)(y) = y(x)$, for all $x \in X$ and $y \in Y$. The reflection map is continuous.

	Suppose $X$ is a separable compact space and $Y$ is a \lind subspace of $C_p(X)$ which does not have a countable network. We consider the reflection map $\refmap{\vphi}{XY}: X \to C_p(Y)$ and let $X_1 = \refmap{\vphi}{XY}(X)$. Then $X_1$ is separable and compact. Next, consider the dual map $\refmap{\vphi}{XY}^*: C_p(X_1) \to C_p(X)$. It's an embedding, so $Y_1 = (\refmap{\vphi}{XY}^*)^{-1}(Y)$ is a subspace of $C_p(X_1)$ homeomorphic to $Y$. Since $Y$ does not have a countable network, neither does $Y_1$. Then neither does $C_p(X_1)$, so neither does $X_1$. But by Lemmas \ref{lem:PFAarhan} and \ref{lem:OR17}, $X_1$ is metrizable. This is a contradiction, since compact metrizable spaces have a countable network.
\end{proof}

Let us mention some  more open problems.

\begin{prob}\label{prob:lind1count.groth}
	Are \lind first countable spaces \groth?
\end{prob}

Although we can't fully answer Problem \ref{prob:lind1count.groth}, we can weaken the hypothesis of Theorem \ref{thm:dramatic} in the first countable case:

\begin{thm}\label{thm:weakdramatic}
	$\MA_{\omega_1}$ implies that every \lind first countable space is \groth.
\end{thm}

Before proving this, we need to mention some more general facts about $C_p$, taken from \cite{Okunev1995}.

\begin{lem}
	Let $Y \subseteq C_p(X)$. Let $x_1, x_2 \in X$. Let $x_1\sim_Y x_2$ if $y(x_1) = y(x_2)$ for all $y \in Y$. Let $X_1$ be the set of equivalence classes and $\pi: X \to X_1$ the natural map. For any $y \in Y$, there is a $y': X_1 \to \reals$ such that $y = y'\com \pi$. Give $X_1$ the weakest topology that makes all of the $y'$'s continuous. With this topology, $X_1$ is homeomorphic to $\refmap{\vphi}{XY}(X)$. Then $(\refmap{\vphi}{XY}^*)^{-1}(Y)$ is a subspace of $C_p(X_1)$ homeomorphic to $Y$. 
\end{lem}

In particular, this tells us that if $K$ is a separable compact subspace of $C_p(Y)$, then $Y$ is homeomorphic to a subspace of $C_p(K_1)$, where $K_1$ is a continuous image of $K$ and hence is separable and compact.

\begin{proof}[Proof of Theorem \ref{thm:weakdramatic}.]
	Let $Y$ be \lind and first countable. Let $K$ be a compact separable subspace of $Y$. Let $K_1$ be a continuous image of $K$, and $Y_1$ be a homeomorphic copy of $Y$ included in $C_p(K_1)$. We now invoke two applications of $\MA_{\omega_1}$:

	\begin{lemr}[\cite{Okunev2007}]
		$\MA_{\omega_1}$ implies that if $K$ is a compact separable space, then every \lind subspace of $C_p(K)$ is hereditarily \lind.
	\end{lemr}

	\begin{lemr}[\cite{Szentmiklossy1980}]
		$\MA_{\omega_1}$ implies that every first countable hereditarily \lind space is hereditarily separable.
	\end{lemr}

	But,

	\begin{lemr}[{\cite[5.26]{Arhangelskii1998}}]
		Every hereditarily separable space is \groth.
	\end{lemr}\vspace{-1.2cm}
\end{proof}

A corollary of what we just proved is of interest. 

\begin{crl}
	$\MA_{\omega_1}$ implies that if $K$ is a compact subspace of $C_p(Y)$, where $Y$ is \lind and first countable, then $K$ is metrizable.
\end{crl}

\begin{proof}
	In the previous proof, we showed $Y$ was hereditarily separable. \arhan proved:

	\begin{lemr}[{\cite[3.13]{Arhangelskii1997a}}]
		If $Y$ is separable and $K$ is a compact subspace of $C_p(Y)$, then $K$ is metrizable.
	\end{lemr}\vspace{-1.2cm}
\end{proof}

In the spirit of Problem \ref{prob:lind1count.groth}, one can ask:

\begin{prob}
	If $X$ is a separable compact space and $Y$ is a \lind first countable subspace of $C_p(X)$, does $Y$ have a countable network?
\end{prob}

We have a partial answer:

\begin{thm}
	$\MA_{\omega_1}$ implies that if $X$ is a separable compact space and $Y$ is a \lind first countable subspace of $C_p(X)$, then $Y$ has a countable network.
\end{thm}

This follows from what we have just done by the same argument as for Theorem \ref{thm:uselemmas}.

Note that:

\begin{thm}
\label{thm:hereditarygspace}
	If $Y$ is a hereditary $g$-space, then countably compact subspaces of $Y$ are compact.
\end{thm}

\begin{proof}
	Let $Z \subseteq Y$ be countably compact. Then it is countably compact in itself and its closure in itself is compact.
\end{proof}

\begin{prob}
	If countably compact subspaces of $C_p(X)$ are compact, is $X$ \groth?
\end{prob}

There is a necessary and sufficient condition on $X$ so that $C_p(X)$ is countably tight (see Lemma \ref{lem:xnctight} below), and there is even a necessary and sufficient condition on $X$ that ensures $C_p(X)$ is \frech \cite{Gerlits1982}, but these conditions are too onerous and entail more than we need.

\begin{prob}
	Find a necessary and sufficient condition on $X$ such that compact subspaces of $C_p(X)$ are countably tight.
\end{prob}

\begin{defi}
	A sequence $\seq{x}{\alpha}{\alpha<\kappa}$ is \emph{free} if for all $\beta < \kappa$, \[\clsr{\seq{x}{\alpha}{\alpha<\beta}} \insect \clsr{\seq{x}{\alpha}{\alpha\geq\beta}} = \emptyset.\]
\end{defi}

It is well-known that:

\begin{lem}
	If $X$ is \lind and countably tight, then $X$ does not include an uncountable free sequence.
\end{lem}

\begin{proof}
	Suppose $F = \seq{x}{\alpha}{\alpha < \omega_1}$ is free. Let $F_\beta = \clsr{\seq{x}{\alpha}{\alpha<\beta}}$. Then $\clsr{\st{F_\beta}{\beta<\omega_1}}$ is a decreasing family of closed subspaces of $X$. Since $X$ is \lind, there is an $x \in \Insect\clsr{\st{F_\beta}{\beta<\omega_1}}$. Then $x \in \clsr{F}$ but $x \notin \clsr{A}$ for any countable $A \subseteq F$, contradicting countable tightness.
\end{proof}

Todorcevic proved:

\begin{thmr}[\cite{TodorcevicSL}]
	$\PFA$ implies: if $X$ includes no uncountable free sequences, then every countably compact subspace of $C_p(X)$ is compact.
\end{thmr}

The hypothesis is weaker than that of Theorem \ref{thm:dramatic}, but so is the conclusion.

\begin{prob}
	Does $\PFA$ imply that if $X$ includes no uncountable free sequences, then $X$ is (weakly) \groth?
\end{prob}

Todorcevic also proved:

\begin{lemr}[\cite{TodorcevicSL}]
	Suppose every countably compact subspace of $C_p(X)$ is compact. Then every compact subspace of $C_p(X)$ is countably tight.
\end{lemr}

\begin{proof}
	Suppose there is a $Z \subseteq C_p(X)$ such that there is a $y \in \clsr{Z} - \Union\st{\clsr{Z_0}}{Z_0 \subseteq Z \mbox{ is countable}}$. But $\Union\st{\clsr{Z_0}}{Z_0 \subseteq Z \mbox{ is countable}}$ is countably compact, hence compact, hence closed, a contradiction.
\end{proof}

\begin{crlr}[\cite{TodorcevicSL}]
	$\PFA$ implies that if $X$ does not include any uncountable free sequences, then compact subspaces of $C_p(X)$ are countably tight.
\end{crlr}

%%%%%%%%%%%%%%%%%%%%%%%%%%%%%%%%%%%%%%%%
%%%%%%%%%%%%%% COUNTEREXAMPLES %%%%%%%%%
%%%%%%%%%%%%%%%%%%%%%%%%%%%%%%%%%%%%%%%%
\section{Counterexamples}

In \cite{Arhangelskii1998}, \arhan asked whether the conclusion of Proposition \ref{thm:MAnoCH} is true in $\ZFC$. It is not:

\begin{eg}
\label{egDiamond}
	Assuming $\diamondsuit$ plus Kurepa's Hypothesis, Ivanov \cite{Ivanov1978} constructs a compact space $Y$ of cardinality $2^\mfrak{c}$ such that $Y^n$ is hereditarily separable for all $n < \omega$. $C_p(Y)$ is the required counterexample.
\end{eg}

To see this, we require several results from the literature.

\begin{lemr}[\cite{ArhangelskiiFunction}]
\label{lem:xnctight}
	$X^n$ is \lind for every $n < \omega$ if and only if $C_p(X)$ is countably tight.
\end{lemr}

\begin{lemr}[\cite{ArhangelskiiFunction}]
	$X$ embeds into $C_p(C_p(X))$.
\end{lemr}

Clearly, separable Fr\'echet-Urysohn spaces have cardinality $\leq \mfrak{c}$. Ivanov's space $Y$ is too big to be Fr\'echet-Urysohn, yet it embeds in $C_p(C_p(Y))$, so $C_p(Y)$ cannot be \groth, although it is weakly \groth. $(C_p(Y))^n$ is, however, (hereditarily) \lind for all $n < \omega$ by the Velichko-Zenor theorem:

\begin{lemr}[\cite{Velichko1981}, \cite{Zenor1980}]
	If $X^n$ is hereditarily separable for all $n < \omega$, then $(C_p(X))^n$ is hereditarily \lind for all $n < \omega$.
\end{lemr}

Ivanov's space also provides counterexamples for various other propositions proved by \arhan, Okunev, and Reznichenko under $\MA_{\omega_1}$ or $\PFA$. $Y$ is \slind, compact, countably tight, separable, but not metrizable. This violates the conclusion of Lemma \ref{lem:OR17}.

\begin{defi}
	An \emph{$S$-space} is a hereditarily separable space that is not hereditarily \lind. A \emph{strong $S$-space} is an $S$-space with all finite powers hereditarily separable.
\end{defi}

\begin{lemr}[\cite{Todorcevic1989}]
	$\mfrak{b} = \aleph_1$ implies there is a compact strong $S$-space.
\end{lemr}

$\mfrak{b} = \aleph_1$ is weaker than $\CH$, which is weaker than $\diamond$. Todorcevic's space will work for the purpose of violating the conclusion of Lemma \ref{lem:OR17} as well as $Y$. (To be more precise, Todorcevic constructs a locally compact, locally countable strong $S$-space $T$, but then its one-point compactification $T^* = T \union \{*\}$ is a compact strong $S$-space).

Both $Y$ and $T^*$, embedded in $C_p(C_p(Y))$ and $C_p(C_p(T^*))$ respectively, provide counterexamples to the conclusion of Proposition \ref{prop:slind.cpct.sep.metrizable}.

They also provide counterexamples to the conclusion of Proposition \ref{prop:slind.cpct.monolithic}. The point is that compact spaces with countable network weight are metrizable.

Any compact strong $S$-space $Y$ refutes the conclusion of Proposition \ref{prop:okunevasserts}. $(C_p(Y))^n$ will be hereditarily \lind, but $C_p(Y)$ does not have a countable network, else $Y$ would, but then $Y$ would be hereditarily \lind.

The hypothesis for Example \ref{egDiamond} seems too strong; $\diamondsuit$ ought to suffice. I conjecture that $\diamondsuit$ implies there is a ``strong Ostaszewski space'', i.e.~a strong $S$-space $X$ which is countably compact, perfectly normal, but not compact. $C_p(X)$ would then have finite products hereditarily \lind, but would not be \groth since $X$ would be embedded in $C_p(C_p(X))$, violating Theorem \ref{thm:hereditarygspace}. Notice we are not using perfect normality, so even $\CH$ might suffice.

\begin{rmk*}
	\cite{Okunev2007} appears to have some misprints. I believe that their reference 2 should actually be our \cite{ArhangelskiiFunction}. Their Question 0.5 is asserted to be essentially the same as Problem IV.1.8 in \cite{ArhangelskiiFunction} but the latter problem apparently has nothing to do with the former, so this may be a misprint. They call ``centered'' what is normally called ``linked'' \cite{Kunen1979}. On the other hand, their use of ``\slind'' for the concept \cite{ArhangelskiiFunction} calls ``suplindel\"of'' is correct. Professor \arhan has informed me that ``suplindel\"of'' was a mistranslation by the translator of the Russian original.
\end{rmk*}

\newpage
\bibliographystyle{alpha}
\bibliography{refdb2020.bib}

\end{document}